\documentclass[reqno]{amsart}
\usepackage{geometry}
\usepackage{graphicx}
\usepackage{color}
\usepackage[all]{xy}
\usepackage{verbatim}
\usepackage{amsmath, amsthm, amssymb, amscd}

\numberwithin{equation}{section}
\newtheorem{thm}[equation]{Theorem}
\newtheorem{prop}[equation]{Proposition}

\theoremstyle{definition}
\newtheorem{dfn}[equation]{Definition}

\newtheorem{rem}[equation]{Remark}
\newtheorem*{acknow}{Acknowledgement}

\newcommand{\nat}{\mathbb{N}}
\newcommand{\ent}{\mathbb{Z}}

\newcommand{\com}{\mathbb{C}}

\newcommand{\fie}{\mathbf{k}}
\newcommand{\q}{\mathbf{q}}


\begin{document}
\title[Deformations of Quantum Symmetric Algebras Extended by Groups]
{Deformations of Quantum Symmetric Algebras Extended by Groups}

\author[Jeanette Shakalli]{Jeanette Shakalli}
\address{Department of Mathematics, Texas A\&M University, College
Station, TX 77840}
\email{jshakall@gmail.com}




\date{\today.}

\begin{abstract}
We discuss a general construction of
a deformation of a smash product algebra coming from an action of a 
particular Hopf algebra. This Hopf algebra is generated by skew-primitive 
and group-like elements, and depends on a complex parameter. The smash 
product algebra is defined on the quantum symmetric algebra of a 
finite-dimensional vector space and a group. In particular, 
an application of this result has enabled us to find a deformation of 
such a smash product algebra which is, to the best of our knowledge, 
the first known example of a deformation in which the new relations in 
the deformed algebra involve elements of the original vector space. 
Using Hochschild cohomology, we show that the resulting
deformations are nontrivial by giving the precise characterization of
the infinitesimal.
\end{abstract}

\maketitle


\section{Introduction} \label{ch1}
A deformation of an algebra is obtained by slightly modifying its multiplicative 
structure. Deformations arise
in many areas of mathematics, such as combinatorics~\cite{MR1314036}, 
representation theory~\cite{MR831053}, and orbifold theory~\cite{MR1881922}. 
From this assertion, it is clear that the study of deformations of an 
algebra is of utmost importance. The main interest in this work
is to study deformations arising from an action of a Hopf algebra.

According to Giaquinto's survey paper~\cite{MR2762537}, 
the works of Fr{\"o}hlicher and Nijenhuis~\cite{MR0083801}, and 
Kodaira and Spencer~\cite{MR0112154} on deformations of complex manifolds
set the basic foundations of the modern theory of deformations. Gerstenhaber's 
seminal paper~\cite{MR0171807} marked the beginning of algebraic deformation 
theory. Although the origins of deformations lie in analytic theory, the 
algebraic setting provides a much more general framework.

In general, as described in~\cite{MR0171807}, finding all possible deformations 
of an algebra is quite a challenging task.
For a certain kind of Hopf algebra, Witherspoon~\cite{MR2267580} has found an explicit formula
that yields a deformation of its Hopf module algebras. 
Applications of 
this formula have been studied in~\cite{GUCCIONE} and~\cite{MR2267580} for the case of 
the smash product algebra of the symmetric algebra of a vector space with a group. 
The purpose of this work is to extend these results to the case of the smash 
product algebra of the \emph{quantum} symmetric algebra of a vector space with a group, 
thus increasing the number of known examples of deformations.

To place this work in context, let us briefly describe what is known about
deformations. First, let us introduce some notation. In this work, the set of
natural numbers $\nat$ includes $0$. By $\fie$, we will denote a field of characteristic $0$. 
Unless stated otherwise, by $\otimes$ we mean $\otimes_\fie$. 
Let $V$ be a $\fie$-vector space with basis $\{w_1,\dots,w_k\}$ and $T(V)$ 
the tensor algebra of $V$.
Let $q_{ij} \in \fie^\times$ for which $q_{ii} = 1$ and
$q_{ji} = q_{ij}^{-1}$ for $i,j = 1, \dots, k$. Set $\q = (q_{ij})$. 
The quantum symmetric algebra $S_\q(V)$ of $V$ is defined as 
\[ S_\q(V) = T(V) \; / \; (w_i w_j - q_{ij} w_j w_i \; | \; 1 \leq i,j \leq k), \]
where the element $w_i \otimes w_j$ is abbreviated as $w_i w_j$.
If $q_{ij} = 1$ for all $i,j$, then we obtain the symmetric algebra $S(V)$. 
Let $G$ be a group acting linearly on $V$. If the action of $G$ on $V$ extends to
all of $S_\q(V)$, then the smash product algebra $S_\q(V)\#G$ is obtained by using 
$S_\q(V) \otimes \fie G$ as a vector space but with a new multiplication given by
\[ (a\#g)(b\#h) = a \; g(b) \; \# \; gh \quad \text{ for all } a,b \in S_\q(V), \; g,h \in G. \]
For further details on this construction, see~\cite{MR1243637}.

Graded Hecke algebras, also known as Drinfeld Hecke algebras~\cite{MR831053},
can be viewed as deformations of $S(V)\#G$ of type 
\begin{equation} \label{graded}
(T(V)\#G) \; / \; (w_i w_j - w_j w_i - \sum_{g \in G} a_g(w_i,w_j) \; g), \quad
\text{ where } a_g(w_i,w_j) \in \fie.
\end{equation} 
Symplectic reflection algebras~\cite{MR1881922} and rational 
Cherednik algebras~\cite{MR1314036} are special cases of Drinfeld Hecke algebras.
Analogously, braided Cherednik algebras~\cite{MR2511188} are deformations of $S_\q(V)\#G$ of type
\begin{equation} \label{braided} 
(T(V)\#G) \; / \; (w_i w_j - q_{ij} w_j w_i - \sum_{g \in G} a_g(w_i,w_j) \; g), \quad
\text{ where } a_g(w_i,w_j) \in \fie. 
\end{equation}
For other types of deformations of $S(V)\#G$ and $S_\q(V)\#G$, not as much is known, 
and from what is known, not all the resulting algebras are necessarily defined by generators 
and relations. In some cases, these algebras are obtained via a universal deformation formula, 
originally introduced by Giaquinto and Zhang in~\cite{MR1624744}, coming from a Hopf algebra action.

In his fundamental work~\cite{MR0171807}, Gerstenhaber showed that 
two commuting derivations on an algebra $A$ lead to a deformation of $A$.
Giaquinto and Zhang generalized this idea in~\cite{MR1624744} with their theory of universal 
deformation formulas. The question that naturally arises is what about skew
derivations. The answer is that sometimes skew derivations do lead to a deformation
via a Hopf algebra action, as shown in~\cite{GUCCIONE} and~\cite{MR2267580} for
the case of $S(V)\#G$. The work presented here generalizes these results to 
the quantum version, that is to the case of $S_\q(V)\#G$. This is particularly 
relevant since explicit examples of deformations of 
$S_\q(V)\#G$ have proven to be difficult to find. Moreover, it turns out that 
some of the deformations constructed in this work are not graded in the sense 
of Braverman and Gaitsgory~\cite{MR1383469} (see Remark~\ref{rem:generalex}). 
We are confident that this newly found set of examples will 
be helpful in the quest to understand deformations more generally.

This paper is organized as follows: In Section~\ref{ch4}, we review the basics
of algebraic deformation theory and define some concepts that we will need.
In Section~\ref{motivation}, we construct some original examples of deformations.
In Section~\ref{general}, we build a general theory that encompasses the examples 
found in Section~\ref{motivation} as a special case. In Section~\ref{ch5}, 
using Hochschild cohomology, we prove that the deformations obtained in 
Section~\ref{general} are nontrivial.

\section{Preliminaries} \label{ch4}
In this section, we define a deformation of an algebra and briefly discuss 
the difficulties that arise when trying to find all possible deformations 
of a given algebra. We end this section by introducing universal deformation 
formulas. The basics of algebraic deformation theory can be found 
in~\cite{MR0171807}, \cite{MR981619} and~\cite{MR2762537}.

\begin{dfn} \label{def:deformation}
Let $t$ be an indeterminate. A \emph{formal deformation} of a 
$\fie$-algebra $A$ is an associative algebra $A[[t]]$ over the formal
power series ring $\fie[[t]]$ with multiplication
\[ a * b = ab + \mu_1(a \otimes b) \; t + \mu_2(a \otimes b) \; t^2 + \dots \]
for all $a,b \in A$, where $ab$ denotes the multiplication in $A$
and the maps $\mu_i : A \otimes A \rightarrow A$ are $\fie$-linear extended 
to be $\fie[[t]]$-linear.
\end{dfn}

Given an algebra $A$, a natural question that arises is to determine all possible deformations of $A$.
By Definition~\ref{def:deformation}, we may restate this problem as follows:
Find all sequences $\{\mu_i\}$ such that the corresponding map $*$ is associative on $A[[t]]$.
This associativity condition allows us to match coefficients of $t^n$ for $n \in \nat$, which in turn
gives conditions that the maps $\mu_i$ must satisfy. It is well-known that the obstructions to the
existence of the maps $\mu_i$ are Hochschild $3$-coboundaries (see, for instance,~\cite{MR0171807}), 
which motivates a recursive
procedure to find the maps $\mu_i$. Such a procedure is a very difficult process, if at all possible.
In fact, this is a potentially infinite process. 
For this reason, many authors, such as Guccione et al.~\cite{GUCCIONE}
and Witherspoon~\cite{MR2267580}, concentrate their efforts in the study of deformations 
of an algebra coming from an action of a Hopf algebra.

Giaquinto and Zhang~\cite{MR1624744} developed the theory of universal deformation formulas. 
Here we define a universal deformation formula based on a bialgebra. In what follows,
by $\Delta$ and $\varepsilon$ we denote the comultiplication and the counit of the bialgebra.

\begin{dfn} \label{def:UDF}
A \emph{universal deformation formula} based on a bialgebra $B$ is
an element $F \in (B \otimes B)[[t]]$ of the form 
\[ F = 1_B \otimes 1_B + t \; F_1 + t^2 \; F_2 + \dots \]
where $F_i \in B \otimes B$, for $i=1,2,\dots$, satisfying the following three conditions:
\begin{align*} 
(\varepsilon \otimes \mathrm{id}_B) \; (F) & = 1 \otimes 1_B, \\
(\mathrm{id}_B \otimes \varepsilon) \; (F) & = 1_B \otimes 1, \\
[ \; (\Delta \otimes \mathrm{id}_B) \; (F) \; ] \; (F \otimes 1_B) & = 
[\; (\mathrm{id}_B \otimes \Delta) \; (F) \;] \; (1_B \otimes F).
\end{align*}
\end{dfn}

Such a formula is universal in the sense that it applies to any 
$B$-module algebra $A$ to yield a deformation of $A$ (see Theorem~$1.3$
in~\cite{MR1624744} for a detailed proof). In particular,
$m \circ F$ is the multiplication in the deformed algebra of $A$, where $m$ 
denotes the multiplication in $A$.

Since it will be needed in what follows, let us recall the Hopf algebra $H_q$ 
introduced in~\cite{MR2267580}.
Let $q \in \com^\times$. Let $H$ be the algebra generated by $D_1$, $D_2$,
$\sigma$ and $\sigma^{-1}$ subject to the relations
\[ D_1 D_2 = D_2 D_1, \quad
 q \sigma D_i = D_i \sigma \; \; \text{ for } i = 1,2, \quad
 \sigma \sigma^{-1} = \sigma^{-1} \sigma = 1. \]
Then $H$ is a Hopf algebra with
\begin{align*}
\Delta(D_1) & = D_1 \otimes \sigma + 1 \otimes D_1,
&\varepsilon(D_1) & = 0,
&S(D_1) & = -D_1 \sigma^{-1}, \\
\Delta(D_2) & = D_2 \otimes 1 + \sigma \otimes D_2,
&\varepsilon(D_2) & = 0, 
&S(D_2) & = -\sigma^{-1} D_2, \\
\Delta(\sigma) & = \sigma \otimes \sigma,
&\varepsilon(\sigma) & = 1,
&S(\sigma) & = \sigma^{-1},
\end{align*}
where $S$ denotes the antipode of $H$.
Let $n \in \ent, \; n \geq 2$. Let $I$ be the ideal of $H$ generated by 
$D_1^n$ and $D_2^n$. If $q$ is a primitive $n$th root of unity, 
then $I$ is a Hopf ideal. Thus, the quotient $H/I$ is also a Hopf algebra. 
Define 
\[ H_q =
\begin{cases}
 H/I, & \text{ if } q \text{ is a primitive } n \text{th root of unity,} \\
 H, & \text{ if } q = 1 \text{ or } q \text{ is not a root of unity.}
\end{cases}
\]

Witherspoon proved in~\cite{MR2267580} that
\[ \mathrm{exp}_q(t \; D_1 \otimes D_2) =
\begin{cases}
 \sum_{i=0}^{n-1} \frac{1}{(i)_q!} \; (t D_1 \otimes D_2)^i, & \text{if } q 
 \text{ is a primitive } n \text{th root of unity,} \\
 \sum_{i=0}^{\infty} \frac{1}{(i)_q!} \; (t D_1 \otimes D_2)^i, & \text{if } 
 q = 1 \text{ or } q \text{ is not a root of unity},
\end{cases} \]
is a universal deformation formula based on $H_q$, where $(i)_q$ denotes
the quantum integer and $(i)_q!$ the quantum factorial. Therefore, for
every $H_q$-module algebra $A$, 
\begin{equation} \label{formula}
 m \circ \mathrm{exp}_q(t \; D_1 \otimes D_2)
\end{equation}
yields a formal deformation of $A$, where $m$ denotes the multiplication
in $A$.
By an $H_q$-module algebra $A$, we mean an
algebra that is also an $H_q$-module, for which the two structures are 
compatible, that is
\begin{equation} \label{conditionstar}
 h(ab) = \sum h_1(a) \; h_2(b) \quad \text{and} \quad h(1_A) = \varepsilon(h) \; 1_A \quad
 \text{ for all } h \in H_q, \; a,b \in A,
\end{equation}
where the sigma notation
\[ \Delta(h) = \sum h_1 \otimes h_2 \; \in \; H_q \otimes H_q \quad \text{ for } h \in H_q \]
is as in~\cite{MR1243637}.

We end this section with the following notation which we will use in Section~\ref{ch5}.
Denote by $V^*$ the dual vector space of $V$ and by $\{w_1^*,\dots,w_k^*\}$ the corresponding dual
basis, i.e. $w_i^*(w_j) = \delta_{ij},$ where $\delta_{ij}$ denotes the Kronecker symbol.
For any $\alpha \in \nat^k$, the length $|\alpha|$ of $\alpha$ is defined as
\[ |\alpha| = \sum_{i=1}^k \alpha_i. \] 
For $\alpha \in \nat^k$, we define $w^\alpha = w_1^{\alpha_1} \cdots w_k^{\alpha_k}$.
Finally, whenever $A$ is a set and $G$ a group acting on $A$, we will denote by $A^G$ the elements
of $A$ that are invariant under this action.

Let $\q = (q_{ij})$ be as in Section~\ref{ch1}. The quantum exterior algebra $\textstyle \bigwedge_\q(V)$
of $V$ is defined as
\[ \textstyle \bigwedge_\q(V) = T(V) \; / \; (w_i w_j + q_{ij} w_j w_i \; | \; 1 \leq i,j \leq k). \]
If $q_{ij} = 1$ for all $i,j$,
then we obtain the exterior algebra $\bigwedge(V)$. Denote the multiplication
in $\textstyle \bigwedge_\q(V)$ by $\wedge$. For any $\beta \in \{0,1\}^k$, let $w^{\wedge \beta}$ 
denote the vector $w_{j_1} \wedge \dots \wedge w_{j_m} \in \bigwedge^m(V)$, which is 
defined by $m = |\beta|, \; \beta_{j_\ell} = 1$ for all $\ell = 1,\dots,m$, and 
$1 \leq j_1 < \dots < j_m \leq k$.

\section{A Motivational Example} \label{motivation}

The example described in this section is, to the best of our knowledge, the
first known example of a deformation of $S_\q(V)\#G$ in which the new relations
in the deformed algebra involve not only group elements and the indeterminate,
but also elements of $V$ (compare to~\eqref{graded} and~\eqref{braided}). 
We also present a generalization of this motivational example, which in turn 
generates numerous new examples of deformations of $S_\q(V)\#G$. 

Let $\fie = \com$ and $\mathrm{dim}(V) = 3$. 
Let $q \in \com$ be a primitive $n$th root of unity for $n \geq 2$.
Define $S_\q(V)$ by setting $q_{12} = q_{13} = q$ and $q_{23} = 1$.
Let \[ G = \langle \sigma_1, \sigma_2 \; | \; \sigma_1^n = \sigma_2^n =
1, \; \sigma_1 \sigma_2 = \sigma_2 \sigma_1 \rangle \] be a group acting on
$S_\q(V)$ via
\begin{align*}
\sigma_1(w_1) & = qw_1, &\sigma_1(w_2) & = w_2, &\sigma_1(w_3) & = qw_3, \\
\sigma_2(w_1) & = w_1, &\sigma_2(w_2) & = qw_2, &\sigma_2(w_3) & = qw_3.
\end{align*}
Define the functions $\chi_i : G \rightarrow \com^\times$ by 
\[ g(w_i) = \chi_i(g) \; w_i \quad \text{ for all } g \in G, \; i = 1,2,3. \]
In particular, $\chi_1(\sigma_1) = q$ and $\chi_1(\sigma_2) = 1$.
Define an action of $H_q$ on the generators of $S_\q(V)$ and on $G$ by
   \begin{align*}
    D_1(w_1) & = \sigma_2, 
    &D_2(w_1) & = 0, 
    &\sigma(w_1) & = q \; w_1, \\
    D_1(w_2) & = 0, 
    &D_2(w_2) & = w_3 \; \sigma_1 \; \sigma_2^{-1}, 
    &\sigma(w_2) & = w_2, \\
    D_1(w_3) & = 0, 
    &D_2(w_3) & = 0,
    &\sigma(w_3) & = w_3, \\
    D_1(g) & = 0, 
    &D_2(g) & = 0,
    &\sigma(g) & = \chi_1(g^{-1}) \; g.
   \end{align*}
Then extend this action of $H_q$ to all of $S_\q(V) \# G$ under the following conditions:
   \begin{align*}
    D_1(a b) & = D_1(a) \; \sigma(b) + a \; D_1(b), \\
    D_2(a b) & = D_2(a) \; b + \sigma(a) \; D_2(b), \\
    \sigma(a b) & = \sigma(a) \; \sigma(b),
   \end{align*}
for all $a,b \in S_\q(V) \# G$. 

\begin{rem} \label{rem:extend}
 The extension conditions follow from the definition of the coproduct in $H_q$. 
 This is done so as to obtain an $H_q$-module algebra, which is shown to be true 
 in the upcoming discussion.
\end{rem}

\begin{prop} \label{extension}
Let $\{w_1^i \; w_2^j \; w_3^m \; g \; | \; i,j,m \in \nat$, $g\in G\}$ denote a basis
of $S_\q(V) \# G$. Then, assuming it is well-defined, the extension of the action of $H_q$ 
on $S_\q(V)$ and $G$ to all of $S_\q(V) \# G$ is given by the following formulas:
\begin{align*}
 D_1(w_1^i \; w_2^j \; w_3^m \; g) & = 
 (i)_q \; q^{j{+}m} \; \chi_1(g^{{-}1}) \; w_1^{i{-}1} \; w_2^j \; w_3^m \; \sigma_2 \; g, \\
 D_2(w_1^i \; w_2^j \; w_3^m \; g) & = 
 (j)_{q^{{-}1}} \; q^i \; w_1^i \; w_2^{j{-}1} \; w_3^{m{+}1} \; 
 \sigma_1 \; \sigma_2^{{-}1} \; g, \\
 \sigma(w_1^i \; w_2^j \; w_3^m \; g) & = 
 q^i \; \chi_1(g^{{-}1}) \; w_1^i \; w_2^j \; w_3^m \; g,
\end{align*}
where a negative exponent of $w_\ell \; (\ell = 1,2,3)$ is interpreted to be $0$.
 In this setting, $S_\q(V) \# G$ is an $H_q$-module algebra.
 \end{prop}

 \begin{proof}
  It is sufficient to check that the relations of $H_q$ are preserved by the generators 
  of $S_\q(V) \# G$, and that the relations of $S_\q(V) \# G$ are preserved by the 
  generators of $H_q$. 
  For the sake of brevity, we omit the details.
 \end{proof}

\begin{rem} \label{rem:derivation}
 By construction, $\sigma$ is an automorphism of $S_\q(V)\#G$, $D_1$ acts as a $\sigma,1$-skew 
 derivation and $D_2$ acts as a $1,\sigma$-skew derivation of $S_\q(V)\#G$.
\end{rem}

As a consequence of Proposition~\ref{extension}, if we set, for instance, 
$q = -1$ (and hence $n = 2$), then by~\eqref{formula}, we get that
\begin{equation} \label{q-1}
   m \circ \mathrm{exp}_q(t D_1 \otimes D_2) = 
   m \circ \left( \sum_{i=0}^{n-1} \frac{1}{(i)_q!} \; (t D_1 \otimes D_2)^i \right) = 
   m \circ (\mathrm{id} \otimes \mathrm{id} + t D_1 \otimes D_2)
\end{equation}
yields a deformation of $S_\q(V)\#G$.

\begin{prop} \label{prop:q-1}
 If $q = -1$, then the deformation of $S_\q(V)\#G$ given by~\eqref{q-1} is
 \[ (T(V)\#G)[[t]] \; / \; (w_1 w_2 + w_2 w_1 + t w_3 \sigma_1, \; w_1 w_3 + w_3 w_1, \; w_2 w_3 - w_3 w_2). \]
\end{prop}

\begin{proof}
 Let us denote by $\mathcal{D}$ the deformation of $S_\q(V)\#G$ obtained by defining 
 the multiplication on $(S_\q(V)\#G)[[t]]$ by~\eqref{q-1}.
 To find the new relations in the deformed algebra $\mathcal{D}$, consider, for example,
 \begin{align*}
  w_1 * w_2 & = (m \circ (\mathrm{id} \otimes \mathrm{id} + t \; D_1 \otimes D_2)) \; (w_1 \otimes w_2) \\
            & = m \; (w_1 \otimes w_2) + m \; (t \; (D_1 \otimes D_2) \; (w_1 \otimes w_2)) \\
            & = w_1 \; w_2 + m \; (t \; (D_1(w_1) \otimes D_2(w_2))) \\
            & = w_1 \; w_2 + m \; (t \; (\sigma_2 \otimes w_3 \; \sigma_1 \; \sigma_2^{-1})) \\
            & = w_1 \; w_2 + t \; \sigma_2 \; w_3 \; \sigma_1 \; \sigma_2^{-1} \\
            & = w_1 \; w_2 - t \; w_3 \; \sigma_1.
 \end{align*}
 Similarly, 
 \begin{align*}
  w_2 * w_1 & = (m \circ (\mathrm{id} \otimes \mathrm{id} + t \; D_1 \otimes D_2)) \; (w_2 \otimes w_1) \\
            & = m \; (w_2 \otimes w_1) + m \; (t \; (D_1 \otimes D_2) \; (w_2 \otimes w_1)) \\
            & = w_2 \; w_1 + m \; (t \; (D_1(w_2) \otimes D_2(w_1))) \\
            & = w_2 \; w_1.
 \end{align*} 
 Thus,
 \[ w_1 * w_2 + w_2 * w_1 = w_1 \; w_2 + w_2 \; w_1 - t \; w_3 \; \sigma_1 = 
     -t \; w_3 \; \sigma_1, \]
 since $w_1 w_2 = - w_2 w_1$ in the original algebra.
 Dropping the $*$ notation, we get that the new relation in $\mathcal{D}$ is
 \[ w_1 w_2 + w_2 w_1 = - t w_3 \sigma_1, \]
 as desired. Similar calculations show that in the deformed algebra $\mathcal{D}$, the following
 relations also hold:
 \[ w_1 w_3 = - w_3 w_1 \quad \text{and} \quad w_2 w_3 = w_3 w_2. \]
 Define an algebra homomorphism 
 \begin{align*}
  \varphi : (T(V)\#G)[[t]] & \rightarrow \mathcal{D} \\
  w_i & \mapsto w_i \\
  g & \mapsto g.
 \end{align*}
 By construction, the map $\varphi$ is surjective since $w_i$ and $g$ are in the image of $\varphi$
 and these elements generate $\mathcal{D}$ by an inductive argument. To see this, recall that by definition,
 $\mathcal{D}$ is $(S_\q(V)\#G)[[t]]$ as a vector space. As a free $\fie[[t]]$-module, $\mathcal{D}$
 has a free generating set $\{w_1^\alpha \; w_2^\beta \; w_3^\gamma \; g \; | \; \alpha, \beta, 
 \gamma \in \nat\}$. It is possible to show by induction on the degree
 $\alpha + \beta + \gamma$ of $w_1^\alpha \; w_2^\beta \; w_3^\gamma \; g$ that this element is in the 
 image of $\varphi$. 
 
 Let $I$ denote the ideal of $(T(V)\#G)[[t]]$ generated by the relations
 \[ w_1 w_2 + w_2 w_1 + t w_3 \sigma_1, \quad w_1 w_3 + w_3 w_1, \quad w_2 w_3 - w_3 w_2. \]
 The calculations presented above show that $I \subseteq \mathrm{ker}(\varphi)$. 
 Note that $(T(V)\#G)[[t]] \; / \; I$ is a filtered algebra, due to the nature of the
 elements of $I$. 

 By an induction argument on the degree, it is possible to show that the elements of the form
 $w_1^\alpha \; w_2^\beta \; w_3^\gamma \; g$, where $\alpha$, $\beta$ and $\gamma$ 
 are nonnegative integers,
 span $(T(V)\#G)[[t]] \; / \; I$ as a free $\fie[[t]]$-module. Thus, the dimension of the
 associated graded algebra of $(T(V)\#G)[[t]] \; / \; I$ in each degree $n$ is at most the 
 number of elements of the form
 $w_1^\alpha \; w_2^\beta \; w_3^\gamma \; g$ with $\alpha + \beta + \gamma = n$. 
 On the other hand, since $I \subseteq \mathrm{ker}(\varphi)$, the map
 \[ (T(V)\#G)[[t]] \; / \; I \; \longrightarrow \; (T(V)\#G)[[t]] \; / \; \mathrm{ker}(\varphi) \]
 is surjective. Thus, the dimension of the associated graded algebra of $(T(V)\#G)[[t]] \; / \; I$ 
 in each degree $n$ is at least the dimension of the associated graded algebra of
 $(T(V)\#G)[[t]] \; / \; \mathrm{ker}(\varphi)$ in degree $n$. Since
 $(T(V)\#G)[[t]] \; / \; \mathrm{ker}(\varphi)$ is isomorphic to $\mathcal{D}$ and we know that
 the elements of the form $w_1^\alpha \; w_2^\beta \; w_3^\gamma \; g$ form a basis of $\mathcal{D}$,
 it follows that the deformation is precisely $(T(V)\#G)[[t]] \; / \; I$.
\end{proof}

As advertised, the vector space element $w_3$ appears in the new relations multiplied by
the indeterminate $t$ and the group element $\sigma_1$. Thus, restricting to the
polynomial algebra $(T(V)\#G)[t]$
and specializing to $t = 1$, this deformation involves relations of type
\[ w_i w_j - q_{ij} w_j w_i - \sum_{g \in G} w_m \; a_g(w_i,w_j) \; g \quad \text{ for some } m. \]
Notice that these relations differ from those used to define the braided 
Cherednik algebras~\eqref{braided} by the presence of the vector space element $w_m$.

\begin{rem} \label{allq}
The arguments presented in the proof of Proposition~\ref{prop:q-1} can be generalized to
any primitive $n$th root of unity $q$ for $n \geq 2$. For example, set $n = 3$. Then by~\eqref{formula}, 
we get that
\begin{equation} \label{q3=1}
   m \circ \mathrm{exp}_q(t \; D_1 \otimes D_2) = 
   m \circ \left(\mathrm{id} \otimes \mathrm{id} + t \; D_1 \otimes D_2 + 
   \frac{1}{1+q} \; (t^2 \; D_1^2 \otimes D_2^2) \right)
\end{equation}
yields a deformation of $S_\q(V)\#G$.
In this case, the deformation can be found to be
\begin{equation} \label{eq:allq}
 (T(V)\#G)[[t]] \; / \; (w_1 w_2 + w_2 w_1 - q t w_3 \sigma_1, \; 
  w_1 w_3 + w_3 w_1, \; w_2 w_3 - w_3 w_2).
\end{equation}
To see this, it is sufficient to notice that $D_i^2(w_j) = 0$ for $i,j = 1,2$.
\end{rem}

\begin{rem} \label{rem:exgraded}
 Notice that the deformation~\eqref{eq:allq} of $S_\q(V)\#G$ is a graded deformation
 in the sense of Braverman and Gaitsgory~\cite{MR1383469}. To see this,
 we assign degree $1$ to the indeterminate $t$ and to the vector space elements $w_i$, 
 and degree $0$ to the group elements $g$. In this way, each of the relations obtained
 are homogeneous of degree $2$. Therefore, the quotient is graded.
\end{rem}

\begin{rem} \label{rem:generalex}
This example can be generalized to higher dimensions as follows:
Let $q$ be a primitive $n$th root of unity $(n \geq 2)$ and let $\mathrm{dim}(V) = k \geq 3$. Set
$q_{1j} = q$ for $j = 2,\dots,k$, and $q_{ij} = 1$ for $i,j = 2,\dots,k$.
Let $G = \langle \sigma_1, \sigma_2 \; | \; \sigma_1^n = \sigma_2^n =
1, \; \sigma_1 \sigma_2 = \sigma_2 \sigma_1 \rangle$ be a group acting on
$S_\q(V)$ via
 \[ \sigma_1(w_1) = q w_1, \quad \sigma_1(w_2) = w_2, \quad \sigma_1(w_3) = q w_3, 
    \quad \cdots \quad \sigma_1(w_k) = q w_k, \]
 \[ \sigma_2(w_1) = w_1, \quad \sigma_2(w_2) = q w_2, \quad \sigma_2(w_3) = q w_3,
    \quad \cdots \quad \sigma_2(w_k) = q w_k. \]
Define the functions $\chi_i : G \rightarrow \com^\times$ by 
\[ g(w_i) = \chi_i(g) \; w_i \quad \text{ for all } g \in G, \; i = 1,\dots,k. \]
Define an action of $H_q$ on the generators of $S_\q(V)$ and on $G$ as follows: \\
Let $\alpha_1, \alpha_2, \dots, \alpha_{k-2}, \beta_1, \beta_2, \dots, \beta_{k-2} \in \nat$. Then define
 \begin{align*}
  D_1(w_1) & = w_3^{\alpha_1 n} \; w_4^{\alpha_2 n} \cdots w_k^{\alpha_{k-2} \, n} \; \sigma_2, &D_2(w_1) & = 0, \\
  D_1(w_2) & = 0, &D_2(w_2) & = w_3^{\beta_1 n + 1} \; w_4^{\beta_2 n} \cdots w_k^{\beta_{k-2} \, n} \; 
  \sigma_1 \; \sigma_2^{-1}, \\
  D_1(w_3) & = 0, &D_2(w_3) & = 0, \\
  &\vdots &&\vdots \\
  D_1(w_k) & = 0, &D_2(w_k) & = 0, \\
  D_1(g) & = 0, &D_2(g) & = 0,
 \end{align*}
and
 \[ \sigma(w_1) = q w_1, \quad
  \sigma(w_2) = w_2, \quad
  \sigma(w_3) = w_3, \quad \cdots \quad
  \sigma(w_k) = w_k, \quad
  \sigma(g) = \chi_1(g^{-1}) \; g. \]

If we extend the action of $H_q$ to all of $S_\q(V) \# G$ under the same conditions as before and denote
a basis of $S_\q(V) \# G$ by 
$\{ w_1^{i_1} \; w_2^{i_2} \dots w_k^{i_k} \; g \; | \; i_1, \dots, i_k \in \nat, \; g \in G\}$, 
then we obtain the following formulas:
\begin{align*}
 D_1(w_1^{i_1} \; w_2^{i_2} \dots w_k^{i_k} \; g) & =
 (i_1)_q \; q^{i_2 + \dots + i_k} \; \chi_1(g^{-1}) \; w_1^{i_1 - 1} \; w_2^{i_2} \; w_3^{i_3 + \alpha_1 n} \dots
 w_k^{i_k + \alpha_{k-2} \, n} \; \sigma_2 \; g, \\
 D_2(w_1^{i_1} \; w_2^{i_2} \dots w_k^{i_k} \; g) & =
 (i_2)_{q^{-1}} \; q^{i_1} \; w_1^{i_1} \; w_2^{i_2 - 1} \; w_3^{i_3 + \beta_1 n + 1} \; w_4^{i_4 + \beta_2 n} \dots
 w_k^{i_k + \beta_{k-2} \, n} \; \sigma_1 \; \sigma_2^{-1} \; g, \\
 \sigma(w_1^{i_1} \; w_2^{i_2} \dots w_k^{i_k} \; g) & =
 q^{i_1} \; \chi_1(g^{-1}) \; w_1^{i_1} \; w_2^{i_2} \dots w_k^{i_k} \; g.
\end{align*}
It is possible to show that $S_\q(V) \# G$ is an $H_q$-module algebra in this case as well. Thus,
again by~\eqref{formula}, $m \circ \mathrm{exp}_q(t \; D_1 \otimes D_2)$ gives a deformation of
$S_\q(V)\#G$. 

If $q = -1$, then as before, $m \circ (\mathrm{id} \otimes \mathrm{id} + t \; D_1 \otimes D_2)$ yields a
deformation which, in this case, can be found to be the quotient of $(T(V)\#G)[[t]]$ by an
ideal that contains the following relations:
\begin{equation} \label{newstuff}
   w_1 w_2 + w_2 w_1 + t \; (-1)^{(\beta_1 + \dots + \beta_{k-2})n} \; w_3^{(\alpha_1+\beta_1)n+1} \; 
   w_4^{(\alpha_2+\beta_2)n} \cdots w_k^{(\alpha_{k-2}+\beta_{k-2})n} \; \sigma_1,
\end{equation}
\[ w_1 w_j + w_j w_1 \quad \text{for } j = 3, \dots, k, \]
\[ w_i w_j - w_j w_i \quad \text{for } i,j = 2, \dots, k. \]
The proof of this statement is analogous to the proof of Proposition~\ref{prop:q-1}. 
For the sake of brevity, we skip the details.
Notice that in this case we do not obtain a graded deformation in the
sense of Braverman and Gaitsgory~\cite{MR1383469} unless 
$\alpha_1 = \dots = \alpha_{k-2} = \beta_1 = \dots = \beta_{k-2} = 0$.
\end{rem}

\section{The General Construction} \label{general}

Let us present some generalizations of the results obtained in~\cite{GUCCIONE}
and~\cite{MR2267580} to the case of $S_\q(V)\#G$. We provide the necessary
and sufficient conditions for $S_\q(V)\#G$ to have the structure of an $H_q$-module
algebra under some assumptions. As a consequence, by applying~\eqref{formula}, 
we obtain more explicit examples of deformations.

\subsection{$H_q$-module Algebra Structures on Arbitrary Algebras} \label{sub1}

Let $A$ be a $\fie$-algebra and let $\sigma, D_1, D_2 : A \rightarrow A$ 
be arbitrary $\fie$-linear maps. By abuse of notation, we identify $\sigma$,
$D_1$ and $D_2$ with the generators of $H_q$ by the same name.  
Then $A$ can be given the structure of an $H_q$-module via a 
(necessarily unique) $\fie$-linear map $H_q \otimes A \rightarrow A$ 
if and only if the maps $\sigma$, $D_1$ and $D_2$ satisfy the following conditions:
\begin{subequations} \label{action}
\begin{equation} \label{eq1}
 \sigma \text{ is a bijective map, }
\end{equation}
\begin{equation} \label{eq2}
 D_1 D_2 = D_2 D_1,
\end{equation}
\begin{equation} \label{eq3}
 q \sigma D_i = D_i \sigma \text{ for } i=1,2, 
\end{equation}
\begin{equation} \label{eq4}
 \text{if } q \text{ is a primitive } n \text{th root of unity, 
 then } D_1^n = D_2^n = 0.
\end{equation}
\end{subequations}

\begin{rem}
 To see why conditions~\eqref{action} are necessary and sufficient, notice that if 
 we are given maps $\sigma$, $D_1$ and $D_2$ as linear operators from $A$ to $A$, there 
 is only one possible way that we can extend this action to all the elements of 
 $H_q$. So if such an extension exists, it is unique. The only question then is 
 whether such an action is well-defined. Since we start with actions of $\sigma$, 
 $D_1$ and $D_2$ on $A$, we would need to check the relations of $H_q$, which are 
 precisely conditions~\eqref{action}.
 \end{rem}

Given $\sigma, D_1$ and $D_2$ for which~\eqref{action} holds, 
the following result determines the conditions that these maps must satisfy so that 
$A$ becomes an $H_q$-module algebra via the map $H_q \otimes A \rightarrow A$.

\begin{thm} \label{thm1}
 Let $\sigma, D_1, D_2 : A \rightarrow A$ be $\fie$-linear maps satisfying~\eqref{action}.
 Then $A$ is an $H_q$-module algebra via the map $H_q \otimes A \rightarrow A$
 if and only if
 \begin{subequations} \label{thm2.4}
 \begin{equation} \label{eq5}
  \sigma(ab) = \sigma(a) \; \sigma(b) \; \text{ for all } a,b \in A,
 \end{equation}
 \begin{equation} \label{eq6}
  D_1(ab) = D_1(a) \; \sigma(b) + a \; D_1(b) \; \text{ for all } a,b \in A,
 \end{equation}
 \begin{equation} \label{eq7}
  D_2(ab) = D_2(a) \; b + \sigma(a) \; D_2(b) \; \text{ for all } a,b \in A.
 \end{equation}
 \end{subequations}
\end{thm}

 \begin{proof}
  See~\cite[Theorem~$2.4$]{GUCCIONE} and set $\alpha$ to be the identity map on $A$.
 \end{proof}


\subsection{$H_q$-module Algebra Structures on Smash Products} \label{sub2}

In what follows, $\mathrm{dim}(V) = k \geq 3$. Assume that the action of $G$
on $V$ is diagonal with respect to the basis $\{w_1, \dots, w_k\}$. Then
there exist maps $\chi_i : G \rightarrow \fie^\times$ such that
\[ g(w_i) = \chi_i(g) \; w_i \quad \text{ for all } g \in G, \; i = 1, \dots, k. \]
Extend the action of $G$ on $V$ to $S_\q(V)$ by algebra automorphisms. 

\begin{thm} \label{thm2}
 Let $\sigma, D_1, D_2 : V \oplus \fie G \rightarrow S_\q(V)\#G$ be $\fie$-linear maps.
 Suppose there exists a group homomorphism $\xi : G \rightarrow \fie^\times$ such that
 \[ \sigma(g) = \xi(g) \; g \quad \text{ for all } g \in G. \]
 Then $\sigma, D_1, D_2$ extend to give $S_\q(V)\#G$ the structure of an $H_q$-module algebra
 if and only if for all $g \in G$, $\ell = 1,2$, $i,j = 1,\dots,k$, 
 the following conditions hold:
\begin{subequations} \label{thm3.5}
\begin{equation} \label{EQ1}
\sigma: V \rightarrow V \text{ is a bijective } \fie G\text{-linear map, }
\end{equation}
\begin{equation} \label{EQ2}
D_1 D_2(w_i) = D_2 D_1(w_i),
\end{equation}
\begin{equation} \label{EQ3}
 D_1 D_2(g) = D_2 D_1(g),
\end{equation}
\begin{equation} \label{EQ4}
q \sigma D_\ell(w_i) = D_\ell \sigma(w_i),
\end{equation}
\begin{equation} \label{EQ5}
q \sigma D_\ell(g) = \xi(g) \; D_\ell(g),
\end{equation}
\begin{equation} \label{EQ6}
\text{if } q \text{ is a primitive } n \text{th root of unity,
then } D_1^n = D_2^n = 0,
\end{equation}
\begin{equation} \label{EQ7}
D_\ell(w_i w_j) = D_\ell(q_{ij} w_j w_i),
\end{equation}
\begin{equation} \label{EQ8}
\sigma(w_i w_j) = \sigma(q_{ij} w_j w_i),
\end{equation}
\begin{equation} \label{EQ9}
D_\ell(g(w_i) g) = D_\ell(g w_i).
\end{equation}
\end{subequations}
\end{thm}

\begin{proof}
 Notice that $V \oplus \fie G$ embeds into $S_\q(V)\#G$, so in order to apply Theorem~\ref{thm1}, 
 first we need to extend the maps $\sigma$, $D_1$ and $D_2$ from $V \oplus \fie G$ to all of $S_\q(V)\#G$
 by requiring conditions~(\ref{thm2.4}) to hold.
 In order to obtain well-defined $\fie$-linear maps from $S_\q(V)\#G$ to $S_\q(V)\#G$, we need to 
 check that the relations of $S_\q(V)\#G$ are satisfied by the generators of $H_q$. These are 
 precisely conditions~\eqref{EQ7}, \eqref{EQ8} and~\eqref{EQ9}. Notice that $\sigma$ automatically
 preserves the relation $g(w_i) g = g w_i$ since by assumption, 
 \[ g(w_i) = \chi_i(g) \; w_i \quad \text{and} \quad \sigma(g) = \xi(g) \; g \quad
    \text{ for all } g \in G, \; i = 1,\dots,k, \] 
 and by~\eqref{EQ1}, $\sigma : V \rightarrow V$ commutes with the action of $G$, that is 
 \[ \sigma(g(w_i)) = g(\sigma(w_i)) \quad \text{ for all } g \in G, \; i=1,\dots,k. \] 
 Next, we need to make sure that the relations of $H_q$ hold, that is items~\eqref{action}. 
 Notice that~\eqref{eq1} is equivalent to~\eqref{EQ1} since $\sigma$ is bijective on $\fie G$ 
 by its definition. Moreover, \eqref{eq4} is precisely~\eqref{EQ6}. Since it is enough to show that the
 relations of $H_q$ are satisfied by the generators of $S_\q(V)\#G$, it follows that~\eqref{eq2} is
 equivalent to~\eqref{EQ2} and \eqref{EQ3}, and~\eqref{eq3} is equivalent to~\eqref{EQ4} and~\eqref{EQ5}.
\end{proof}

\begin{rem}
 This result was obtained for the case of $S(V)\#G$ in~\cite{GUCCIONE}. Thus,
 Theorem~\ref{thm2} is a generalization of~\cite[Theorem~$3.5$]{GUCCIONE} to $S_\q(V)\#G$ 
 with $\alpha$ being the identity map on $S_\q(V)\#G$, $s$ the twist map, 
 $\chi_\alpha(g) = 1$ and $\chi_\varsigma(g) = \chi_1(g^{-1})$ for all $g \in G$.
\end{rem}

\subsection{A Special Case of $H_q$-module Algebra Structures on 
Smash Products} \label{sub3}

The following result makes it easier to come up with more explicit examples of deformations
of $S_\q(V)\#G$. The setting is as in Subsection~\ref{sub2}.

\begin{thm} \label{thm3}
Let $\sigma, D_1, D_2 : V \oplus \fie G \rightarrow S_\q(V)\#G$ be $\fie$-linear maps.
Assume that the action of $\sigma$ on $V$ is diagonal with respect to the basis 
$\{w_1, \dots, w_k\}$. Let $\lambda_i(\sigma) \in \fie^\times$ be defined by
\[ \sigma(w_i) = \lambda_i(\sigma) \; w_i \quad \text{ for all } i = 1,\dots,k. \]
Suppose there exists a group homomorphism $\xi : G \rightarrow \fie^\times$ such that
\[ \sigma(g) = \xi(g) \; g \; \text{ for all } g \in G. \]
Choose $P_1, P_2 \in S_\q(V)$ such that there are scalars $q_{P_i,w_j}$ satisfying
the following equation:
\[ P_i \; w_j = q_{P_i,w_j} \; w_j \; P_i \; \text{ for all } i \neq j. \]
Assume that
\[ D_1(w_1) = P_1 \; g_1, \quad D_1(w_i) = 0 \; \text{ for all } i \neq 1, 
   \quad D_1(g) = 0 \; \text{ for all } g \in G, \]
\[ D_2(w_2) = P_2 \; g_2, \quad D_2(w_i) = 0 \; \text{ for all } i \neq 2,
   \quad D_2(g) = 0 \; \text{ for all } g \in G, \]
with $g_1, g_2 \in G$.
Then there is an $H_q$-module algebra structure on $S_\q(V)\#G$, for which $\sigma, D_1, D_2$
act as the above chosen maps, if and only if
\begin{subequations} \label{thm3.6}
\begin{equation} \label{EQU2}
q_{P_1,w_i} = q_{1i} \; \lambda_i^{-1}(\sigma) \; \chi_i(g_1^{-1}) \; \text{ for all } i \neq 1,
\end{equation}
\begin{equation} \label{EQU3}
q_{P_2,w_i} = q_{2i} \; \lambda_i(\sigma) \; \chi_i(g_2^{-1}) \; \text{ for all } i \neq 2,
\end{equation}
\begin{equation} \label{EQU4}
g_1 \text{ and } g_2 \text{ belong to the center of } G,
\end{equation}
\begin{equation} \label{EQU5}
g(P_1) = \chi_1(g) \; \xi(g) \; P_1 \text{ for all } g \in G,
\end{equation}
\begin{equation} \label{EQU6}
g(P_2) = \chi_2(g) \; \xi(g^{-1}) \; P_2 \text{ for all }g \in G,
\end{equation}
\begin{equation} \label{EQU7}
P_1 \in \mathrm{ker}(D_2) \text{ and } P_2 \in \mathrm{ker}(D_1),
\end{equation}
\begin{equation} \label{EQU8}
\sigma(P_i) = q^{-1} \; \lambda_i(\sigma) \; \xi(g_i^{-1}) \; P_i \text{ for } i = 1,2,
\end{equation}
\begin{equation} \label{EQU9}
\text{if } q \text{ is a primitive } n \text{th root of unity, then } D_1^n = D_2^n = 0.
\end{equation}
\end{subequations}
\end{thm}

\begin{proof}
 Notice that condition~\eqref{EQU9} is exactly~\eqref{EQ6}. 
 Since $\sigma(w_i) = \lambda_i(\sigma) w_i$ for all $i=1,\dots,k,$
 where $\lambda_i(\sigma) \in \fie^\times$, it follows that $\sigma : V \rightarrow V$ is 
 bijective. Since, in addition, $g(w_i) = \chi_i(g) w_i$ for all $g \in G, \; i = 1,\dots,k$, 
 we have that $\sigma : V \rightarrow V$ is a $\fie G$-linear map. Thus, condition~\eqref{EQ1} 
 is satisfied.
 Also note that condition~\eqref{EQ8} is satisfied by the assumption that
 $\sigma(w_i) = \lambda_i(\sigma) w_i$ for all $i=1,\dots,k$.
 Since by assumption $D_1(g) = D_2(g) = 0$ for all $g \in G$, conditions~\eqref{EQ3} and~\eqref{EQ5} hold.

 We claim that conditions~\eqref{EQU2} and~\eqref{EQU3} are equivalent to~\eqref{EQ7}. For $i \neq 1$,
 \begin{align*}
 D_1(w_1 \; w_i) 
 & = D_1(w_1) \; \sigma(w_i) + w_1 \; D_1(w_i) = D_1(w_1) \; \sigma(w_i) = P_1 \; g_1 \; \lambda_i(\sigma) \; w_i \\
 & = \lambda_i(\sigma) \; P_1 \; g_1(w_i) \; g_1 = \lambda_i(\sigma) \; \chi_i(g_1) \; P_1 \; w_i \; g_1
 \end{align*}
 and
 \begin{align*}
 D_1(q_{1i} \; w_i \; w_1) 
 & = q_{1i} \; D_1(w_i) \; \sigma(w_1) + q_{1i} \; w_i \; D_1(w_1) = q_{1i} \; w_i \; D_1(w_1) \\
 & = q_{1i} \; w_i \; P_1 \; g_1 = q_{1i} \; q_{P_1,w_i}^{-1} \; P_1 \; w_i \; g_1
 \end{align*}
 Therefore, $D_1(w_1 \; w_i) = D_1(q_{1i} \; w_i \; w_1)$ is equivalent to 
 \[ q_{P_1,w_i} = q_{1i} \; \lambda_i^{-1}(\sigma) \; \chi_i(g_1^{-1}) \; \text{ for all } i \neq 1. \]
 Similarly, for $i \neq 2$,
 \begin{align*}
 D_2(w_2 \; w_i) 
 & = D_2(w_2) \; w_i + \sigma(w_2) \; D_2(w_i) = D_2(w_2) \; w_i = P_2 \; g_2 \; w_i \\
 & = P_2 \; g_2(w_i) \; g_2 = \chi_i(g_2) \; P_2 \; w_i \; g_2
 \end{align*}
 and
 \begin{align*}
 D_2(q_{2i} \; w_i \; w_2)
 & = q_{2i} \; D_2(w_i) \; w_2 + q_{2i} \; \sigma(w_i) \; D_2(w_2) = q_{2i} \; \sigma(w_i) \; D_2(w_2) \\
 & = q_{2i} \; \lambda_i(\sigma) \; w_i \; P_2 \; g_2 = q_{2i} \; \lambda_i(\sigma) \; q_{P_2,w_i}^{-1} \; P_2 \; w_i \; g_2
 \end{align*}
 Therefore, $D_2(w_2 \; w_i) = D_2(q_{2i} \; w_i \; w_2)$ is equivalent to
 \[ q_{P_2,w_i} = q_{2i} \; \lambda_i(\sigma) \; \chi_i(g_2^{-1}) \; \text{ for all } i \neq 2. \]

 We claim that conditions~\eqref{EQU4}, \eqref{EQU5} and~\eqref{EQU6} are equivalent to~\eqref{EQ9}.
 Since $D_1(g) = 0$ for all $g \in G$, we have
 \[ D_1(g(w_i) \; g) 
    = D_1(g(w_i)) \; \sigma(g) + g(w_i) \; D_1(g)
    = D_1(g(w_i)) \; \sigma(g)
    = \chi_i(g) \; \xi(g) \; D_1(w_i) \; g \]
 and
 \[ D_1(g \; w_i) = D_1(g) \; \sigma(w_i) + g \; D_1(w_i) = g \; D_1(w_i). \]
 Thus, it is enough to show that 
 \[ \chi_i(g) \; \xi(g) \; D_1(w_i) \; g = g \; D_1(w_i) \]
 holds for all $g \in G, \; i = 1,\dots,k,$
 if and only if conditions~\eqref{EQU4} and~\eqref{EQU5} are satisfied. \\
 There are two possible cases, namely $i = 1$ or $i \neq 1$. If $i = 1$, then
 \[ \chi_1(g) \; \xi(g) \; D_1(w_1) \; g = g \; D_1(w_1) \] 
 simplifies to
 \[ \chi_1(g) \; \xi(g) \; P_1 \; g_1 \; g = g \; P_1 \; g_1. \] 
 This is equivalent to
 \[ \chi_1(g) \; \xi(g) \; P_1 \; g_1 \; g = g(P_1) \; g \; g_1, \] 
 which holds if and only if $g_1$ belongs to the center of $G$ and 
 $g(P_1) = \chi_1(g) \; \xi(g) \; P_1$ for all $g \in G$.
 If $i \neq 1$, then $D_1(w_i) = 0$ so both sides of the equation 
 are equal to zero.

 Similarly, since $D_2(g) = 0$ for all $g \in G$, we have
 \[ D_2(g(w_i) \; g) = D_2(g(w_i)) \; g + \sigma(g(w_i)) \; D_2(g)
    = D_2(g(w_i)) \; g = \chi_i(g) \; D_2(w_i) \; g \]
 and
 \[ D_2(g \; w_i) = D_2(g) \; w_i + \sigma(g) \; D_2(w_i) = \sigma(g) \; D_2(w_i) = 
    \xi(g) \; g \; D_2(w_i). \]
 Thus, it is enough to show that 
 \[ \chi_i(g) \; D_2(w_i) \; g = \xi(g) \; g \; D_2(w_i) \]
 holds for all $g \in G, \; i = 1,\dots,k,$
 if and only if conditions~\eqref{EQU4} and~\eqref{EQU6} are satisfied. \\
 Again, there are two possible cases, namely $i = 2$ or $i \neq 2$. If $i = 2$, then
 \[ \chi_2(g) \; D_2(w_2) \; g = \xi(g) \; g \; D_2(w_2) \]
 simplifies to
 \[ \chi_2(g) \; P_2 \; g_2 \; g = \xi(g) \; g \; P_2 \; g_2. \]
 This is equivalent to
 \[ \xi(g^{-1}) \; \chi_2(g) \; P_2 \; g_2 \; g = g(P_2) \; g \; g_2, \]
 which holds if and only if $g_2$ belongs to the center of $G$ and 
 $g(P_2) = \xi(g^{-1}) \; \chi_2(g) \; P_2$ for all $g \in G$.
 If $i \neq 2$, then $D_2(w_i) = 0$ so both sides of the equation 
 are equal to zero.

 We claim that condition~\eqref{EQU7} is equivalent to~\eqref{EQ2}.
 To see this, notice that $D_1 D_2(w_i) = D_2 D_1(w_i)$ is trivially satisfied for 
 $w_i \in \mathrm{ker}(D_1) \cap \mathrm{ker}(D_2)$. Since $D_1 D_2(w_1) = 0$ and
 \[ D_2 D_1(w_1) = D_2(P_1 \; g_1) = D_2(P_1) \; g_1 + \sigma(P_1) \; D_2(g_1) =  D_2(P_1) \; g_1, \]
 we have that $D_1 D_2(w_1) = D_2 D_1(w_1)$ holds if and only if $P_1 \in \mathrm{ker}(D_2)$. \\
 Similarly, since $D_2 D_1(w_2) = 0$ and
 \[ D_1 D_2(w_2) =  D_1(P_2 \; g_2) = D_1(P_2) \; \sigma(g_2) + P_2 \; D_1(g_2) = D_1(P_2) \; \sigma(g_2), \]
 we have that $D_1 D_2(w_2) = D_2 D_1(w_2)$ holds if and only if $P_2 \in \mathrm{ker}(D_1)$.

 We claim that condition~\eqref{EQU8} is equivalent to~\eqref{EQ4}.
 To see this, consider the following: For $i=1,2$, we have
 \[ D_i \sigma(w_i) = D_i(\lambda_i(\sigma) \; w_i)
    = \lambda_i(\sigma) \; D_i(w_i)
    = \lambda_i(\sigma) \; P_i \; g_i \]
 and
 \[ q \sigma D_i(w_i)
    = q \; \sigma(P_i \; g_i)
    = q \; \sigma(P_i) \; \sigma(g_i)
    = q \; \xi(g_i) \; \sigma(P_i) \; g_i. \]
 Therefore, $D_i \sigma(w_i) = q \sigma D_i(w_i)$ holds if and only if 
 $\sigma(P_i) = q^{-1} \; \lambda_i(\sigma) \; \xi(g_i^{-1}) \; P_i$ for $i=1,2$.
 On the other hand, for $i \neq 1,2$, we have $D_1(w_i) = 0 = D_2(w_i)$ so 
 both sides of the equation 
 \[ D_j \sigma(w_i) = q \sigma D_j(w_i) \; \text{ for } j=1,2, \; i = 1,\dots,k, \]
 are zero.
\end{proof}

\begin{rem}
 The idea behind this result comes from~\cite{GUCCIONE}, which deals with the case 
 of $S(V)\#G$. Thus, Theorem~\ref{thm3.6} is a generalization 
 of~\cite[Theorem~$3.6$]{GUCCIONE} to $S_\q(V)\#G$ with $f(g,h) = 1$ for all $g,h \in G$, 
 $\lambda_{1g} = \chi_1(g)$, $\lambda_{2g} = \chi_2(g)$ for all $g \in G$, $\nu_1 = 1$, 
 and $\nu_2 = 1$, but with new assumptions on $\sigma$, $P_1$ and $P_2$.
\end{rem}

\begin{rem} \label{rem:exthm}
 If we set 
 \[ k = 3, \; P_1 = 1, \; P_2 = w_3, \; g_1 = \sigma_2, \; g_2 = \sigma_1 \; \sigma_2^{-1}, \]
 \[ \lambda_1(\sigma) = q, \; \lambda_2(\sigma) = \lambda_3(\sigma) = 1, \]
 \[ q_{P_1,w_2} = q_{P_1,w_3} = q_{P_2,w_3} = 1, \; q_{P_2,w_1} = q^{-1}, \] 
 then we can see that we can apply Theorem~\ref{thm3} to the motivational example 
 presented in Section~\ref{ch4}.\ref{motivation}. We may also apply Theorem~\ref{thm3} 
 to the generalization of the motivational example constructed in Remark~\ref{rem:generalex} 
 if we set 
 \[ P_1 = w_3^{\alpha_1 n} \; w_4^{\alpha_2 n} \cdots w_k^{\alpha_{k-2} n}, \;
 P_2 = w_3^{\beta_1 n + 1} \; w_4^{\beta_2 n} \cdots w_k^{\beta_{k-2} n}, \;
 g_1 = \sigma_2, \; g_2 = \sigma_1 \; \sigma_2^{-1}, \] 
 \[ \lambda_1(\sigma) = q, \; \lambda_i(\sigma) = 1 \; \text{ for } i=2,\dots,k, \]
 \[ q_{P_1,w_i} = 1 \; \text{ for } i=2,\dots,k, \]
 \[ q_{P_2,w_1} = q^{-(\beta_1 n + 1 + \beta_2 n + \dots + \beta_{k-2} n)}, \; 
    q_{P_2,w_i} = 1 \; \text{ for } i=3,\dots,k. \]
\end{rem}

Recall the multiplication in the deformed algebra given in Definition~\ref{def:deformation}.
If, in the setting of Theorem~\ref{thm3}, we want to find, for instance, an expression for 
the map $\mu_1$, then we can proceed as follows: By~\eqref{formula}, we have that
$m \circ \mathrm{exp}_q(t \; D_1 \otimes D_2)$
gives a deformation of $S_\q(V)\#G$, where $m$ is the multiplication in $S_\q(V)\#G$. 
As a consequence, 
\begin{equation} \label{mu1}
 \mu_1 = m \circ (D_1 \otimes D_2).
\end{equation}

\begin{rem} \label{rem:simplify}
 The assumptions made in Theorem~\ref{thm2} and Theorem~\ref{thm3} are mainly for simplifying purposes.
 We believe that it might be possible to prove these same results in a more general setting. This is a 
 subject for future research.
\end{rem}

\section{Nontriviality of the Deformations} \label{ch5}
Once we have found a formal deformation of an algebra, it may turn out that
this deformation is trivial, in the sense that it is isomorphic to the
formal power series ring with coefficients in the original algebra. To verify
that we have indeed obtained a new object, we must prove that such an 
isomorphism does not exist, which may be difficult to do directly.
Using Hochschild cohomology makes this easier to show.

The main goal of this section is to prove that the deformations of $S_\q(V)\#G$
obtained in Section~\ref{general} are not isomorphic to $(S_\q(V)\#G)[[t]]$. We begin
by briefly describing the connection between algebraic deformation theory and Hochschild
cohomology. Then we give the precise characterization of the infinitesimal of 
the deformations that result from Theorem~\ref{thm3} and~\eqref{formula}. 
As we will see, this characterization suffices to prove the nontriviality of 
the resulting deformations.

\subsection{Deformations and Hochschild Cohomology} \label{sub5.1}

The deformations of any algebra are intimately related to its Hochschild
cohomology. According to~\cite{MR2762537}, Gerstenhaber's works~\cite{MR0161898} 
and~\cite{MR0171807} marked the beginning 
of the study of the connection between Hochschild cohomology and 
algebraic deformation theory. He showed that for an algebra $A$, the 
space $\mathrm{HH}^i(A)$ with $i \leq 3$ has a natural
interpretation related to the maps $\mu_i$ and the obstructions to the existence
of the maps $\mu_i$. In particular, the map $\mu_1$ is a Hochschild $2$-cocycle 
and the obstructions to the existence of the maps $\mu_i$ for $i \geq 2$ 
are Hochschild $3$-cocycles.

The relation between (quantum) Drinfeld Hecke algebras and Hochschild 
cohomology is discussed in~\cite{MR2395161} and~\cite{NAIDUWITH}. 
Roughly speaking, quantum Drinfeld Hecke algebras are generalizations 
of Drinfeld Hecke algebras in which polynomial rings are replaced by 
quantum polynomial rings (see~\cite{NAIDUWITH} and~\cite{LEVANDOVSKYY}
for a precise definition).

As shown in~\cite[Section $4$]{MR2762537}, if we can prove that 
the Hochschild cocycle $\mu_1$ represents a nonzero element in 
the Hochschild cohomology ring, i.e.\ it is not a coboundary, then this 
automatically implies that the deformation is nontrivial.

\subsection{The Precise Characterization of the Infinitesimal} \label{sub5.2}

For the rest of this section, we will work in the setting of 
Subsection~\ref{sub3}. We will assume that all the
conditions necessary for Theorem~\ref{thm3} to hold are fulfilled. We will
also assume that the group $G$ is finite and that $P_1$ and $P_2$ are elements
of the subalgebra of $S_\q(V)$ generated by $\{ w_3, \dots, w_k\}$. Notice 
that the examples presented in Section~\ref{motivation} satisfy 
these assumptions.

Recall that at the end of Section~\ref{general}, we found $\mu_1$ in the setting 
of Theorem~\ref{thm3} (see~\eqref{mu1}). In this section, we will give a
direct verification that in this case $\mu_1$ is indeed a Hochschild 
$2$-cocycle by identifying which Hochschild $2$-cocycle $\mu_1$ 
is in relation to a known calculation of the Hochschild cohomology ring of 
$S_\q(V)\#G$. Finally, we will prove that this Hochschild $2$-cocycle is 
nonzero in the Hochschild cohomology ring. As a consequence, the deformations 
found using Theorem~\ref{thm3} and~\eqref{formula} are nontrivial.

As explained in Remark~\ref{rem:derivation}, $D_1$ is a $\sigma,1$-skew 
derivation, and $D_2$ is a $1,\sigma$-skew derivation of $S_\q(V)\#G$. Thus, 
$D_1$ and $D_2$ are Hochschild $1$-cocycles. 
From this it can be directly derived that the infinitesimal $\mu_1$, 
defined in~\eqref{mu1}, is a Hochschild $2$-cocycle for $S_\q(V)\#G$, that is $\mu_1$
satisfies
\[ a \; \mu_1(b,c) + \mu_1(a,bc) = \mu_1(ab,c) + \mu_1(a,b) \; c \quad
   \text{ for all } a,b,c \in S_\q(V)\#G. \]

 The following Hochschild cohomology was computed in~\cite{NAIDU}:
 \begin{equation} \label{eq:piyush}
  \mathrm{HH}^2(S_\q(V), S_\q(V) \# G) \cong \bigoplus_{g \in G} 
  \bigoplus_{\substack{\beta \in \{0,1\}^k \\ |\beta| = 2}} 
  \bigoplus_{\substack{\alpha \in \nat^k \\ \alpha-\beta \in C_g}}
  \mathrm{span}_\fie \{ (w^\alpha \# g) \otimes (w^*)^{\wedge \beta} \} 
 \end{equation}
 as a subspace of $S_\q(V)\#G \otimes \bigwedge_{\q^{-1}}(V^*)$, where $C_g$ is defined to be
 \begin{equation} \label{eq:piyushcg} 
  C_g = \left\{ \gamma \in (\nat \cup \{-1\})^k \; \bigg{|} \; \text{ for every } i = 1,\dots,k, \; 
  \prod_{j=1}^k q_{ij}^{\gamma_j} = \chi_i(g) \; \text{ or } \; \gamma_i = -1 \right\} 
 \end{equation}
 for $g \in G$.
 The following result shows that 
 $P_1 \; g_1 \; P_2 \; g_2 \otimes w_1^* \wedge w_2^*$
 is a Hochschild $2$-cocycle. Later, in Proposition~\ref{prop5.6}, we will show that in fact 
 $P_1 \; g_1 \; P_2 \; g_2 \otimes w_1^* \wedge w_2^*$
 may be identified with our Hochschild $2$-cocycle $\mu_1$ defined in~\eqref{mu1}.

\begin{prop}
 The element $P_1 \; g_1 \; P_2 \; g_2 \otimes w_1^* \wedge w_2^*$ is a 
 representative of an element of $\mathrm{HH}^2(S_\q(V) \# G)$.
\end{prop}

\begin{proof}
 By~\cite[Theorem $4.7$]{NAIDU}, it suffices to show that 
 $P_1 \; g_1 \; P_2 \; g_2 \otimes w_1^* \wedge w_2^*$ is 
 a representative of an element of $\mathrm{HH}^2(S_\q(V), S_\q(V)\# G)$
 and is invariant under the action of $G$.
 For the first part, consider the following simple calculation: 
 By~\eqref{EQU6}, we have that
 \[ P_1 \; g_1 \; P_2 \; g_2 = (P_1 \# g_1) \; (P_2 \# g_2)
    = P_1 \; g_1(P_2) \, \# \, g_1 \; g_2
    = \xi(g_1^{-1}) \; \chi_2(g_1) \; P_1 \; P_2 \, \# \, g_1 \; g_2. \]
 Thus, set $g = g_1 \; g_2$ in~\eqref{eq:piyushcg}. 
 The elements $P_1$ and $P_2$ are linear combinations of monomials of the form 
 $w_1^{\rho_1} \; w_2^{\rho_2} \cdots w_k^{\rho_k}$ and $w_1^{\delta_1} \; w_2^{\delta_2} \cdots w_k^{\delta_k}$, 
 respectively, for some $\rho_\ell, \delta_\ell \in \nat$, $\ell = 1,\dots,k$.
 However, since any calculations can be done term-by-term, it suffices to work with just the monomials.
 Thus, we have
 \[ P_1 \; P_2 = w_1^{\rho_1} \cdots w_k^{\rho_k} \; w_1^{\delta_1} \cdots w_k^{\delta_k}. \]
 Set $\alpha = (\rho_1 + \delta_1, \dots, \rho_k + \delta_k)$ and $\beta = (1,1,0,\dots,0)$ 
 in~\eqref{eq:piyush}. To show that 
 \[ P_1 \; g_1 \; P_2 \; g_2 \otimes w_1^* \wedge w_2^* \in \mathrm{HH}^2(S_\q(V), S_\q(V)\#G), \]
 we need to prove that $\alpha - \beta \in C_g$ with $g = g_1 \; g_2$, where
 \begin{align*}
  \alpha - \beta & = (\rho_1 + \delta_1, \dots, \rho_k + \delta_k) - (1,1,0,\dots,0) \\
                 & = (\rho_1 + \delta_1 - 1, \rho_2 + \delta_2 - 1, \rho_3 + \delta_3, \dots, \rho_k + \delta_k).
 \end{align*}
 That is, we want to show that for every $i=1,\dots,k$, the following holds:
 \[ q_{i1}^{\rho_1 + \delta_1 -1} \; q_{i2}^{\rho_2 + \delta_2-1} \; q_{i3}^{\rho_3+\delta_3} \cdots q_{ik}^{\rho_k + \delta_k} = 
    \chi_i(g_1) \; \chi_i(g_2) \quad \text{ or } \quad \gamma_i = -1. \]
 Notice that when $i=1,2$, $\rho_i = \delta_i = 0$, and therefore, $\gamma_i = -1$.
 Otherwise, consider the following:
 \begin{align*}
  P_1 \; P_2 \; w_i
  & = w_1^{\rho_1} \cdots w_k^{\rho_k} \; w_1^{\delta_1} \cdots w_k^{\delta_k} \; w_i \\
  & = \left( \prod_{j=1}^k q_{ji}^{\delta_j} \right) w_1^{\rho_1} \cdots w_k^{\rho_k} \; w_i \; w_1^{\delta_1} \cdots w_k^{\delta_k} \\
  & = \left( \prod_{j=1}^k q_{ji}^{\rho_j} \right) \left( \prod_{j=1}^k q_{ji}^{\delta_j} \right) 
      w_i \; w_1^{\rho_1} \cdots w_k^{\rho_k} \; w_1^{\delta_1} \cdots w_k^{\delta_k} \\
  & = \left( \prod_{j=1}^k q_{ji}^{\rho_j + \delta_j} \right) w_i \; P_1 \; P_2.
 \end{align*}
 By~\eqref{EQU2} and~\eqref{EQU3}, we also have that for $i \neq 1,2$,
 \begin{align*}
  P_1 \; P_2 \; w_i
  & = P_1 \; q_{P_2, w_i} \; w_i \; P_2 = q_{2i} \; \lambda_i(\sigma) \; \chi_i(g_2^{-1}) \; P_1 \; w_i \; P_2 \\
  & = q_{2i} \; \lambda_i(\sigma) \; \chi_i(g_2^{-1}) \; q_{P_1, w_i} \; w_i \; P_1 \; P_2 \\
  & = q_{2i} \; \lambda_i(\sigma) \; \chi_i(g_2^{-1}) \; q_{1i} \; \lambda_i^{-1}(\sigma) \; \chi_i(g_1^{-1}) \; w_i \; P_1 \; P_2 \\
  & = q_{1i} \; q_{2i} \; \chi_i(g_1^{-1}) \; \chi_i(g_2^{-1}) \; w_i \; P_1 \; P_2.
 \end{align*}
 Therefore,
 \[ \prod_{j=1}^k q_{ji}^{\rho_j + \delta_j} = q_{1i} \; q_{2i} \; \chi_i(g_1^{-1}) \; \chi_i(g_2^{-1}). \]
 Or equivalently,
 \[ q_{i1}^{-1} \; q_{i2}^{-1} \; \left( \prod_{j=1}^k q_{ij}^{\rho_j + \delta_j} \right) = \chi_i(g_1) \; \chi_i(g_2) \quad \text{for } i \neq 1,2, \]
 which is exactly what we wanted to show.

 For the second part, to prove that $P_1 \; g_1 \; P_2 \; g_2 \otimes w_1^* \wedge w_2^*$ is 
 invariant under the action of $G$, consider the following calculation:
 \begin{align*}
  g(P_1 \; g_1 \; P_2 \; g_2 \otimes w_1^* \wedge w_2^*) & =
    g(P_1) \; g_1 \; g(P_2) \; g_2 \otimes g(w_1^*) \wedge g(w_2^*) \\ & =
    \chi_1(g) \; \xi(g) \; P_1 \; g_1 \; \chi_2(g) \; \xi(g^{-1}) \; P_2 \; g_2 \otimes 
    \chi_1^{-1}(g) \; w_1^* \wedge \chi_2^{-1}(g) \; w_2^* \\
    & = P_1 \; g_1 \; P_2 \; g_2 \otimes w_1^* \wedge w_2^*
 \end{align*}

 for all $g \in G$, where we used~\eqref{EQU5}, \eqref{EQU6} and
 $g(w_i^*) = \chi_i^{-1}(g) \; w_i^*$.
\end{proof}

 Let us recall the bar resolution of $S_\q(V)$, which we will need in what follows.
 Let $(S_\q(V))^e = S_\q(V) \otimes (S_\q(V))^{\mathrm{op}}$, where $(S_\q(V))^{\mathrm{op}}$ denotes
 the opposite algebra of $S_\q(V)$ as in~\cite{MR1243637}.
 For each $g \in G$, we denote by $(S_\q(V))_g$ the set $\{ cg \; | \; c \in S_\q(V) \}.$
 Notice that $(S_\q(V))_g$ is a left $(S_\q(V))^e$-module via the action
 \[ (a \otimes b) (cg) = ac \; g(b) \; g \quad \text{ for all } a,b,c \in S_\q(V), \; g \in G. \]
 The following is a free $(S_\q(V))^e$-resolution of $S_\q(V)$:
 \[ \cdots \longrightarrow (S_\q(V))^{\otimes 4} \overset{\delta_2}{\longrightarrow}
    (S_\q(V))^{\otimes 3} \overset{\delta_1}{\longrightarrow} (S_\q(V))^e
    \overset{m}{\longrightarrow} S_\q(V) \longrightarrow 0, \]
 where
 \[ \delta_i(a_0 \otimes a_1 \otimes \dots \otimes a_{i+1}) = 
    \sum_{j=0}^i (-1)^j a_0 \otimes \dots \otimes a_j a_{j+1} \otimes \dots \otimes a_{i+1}. \]

 Next, let us introduce the quantum Koszul resolution of $S_\q(V)$, which is again a free 
 $(S_\q(V))^e$-resolution of $S_\q(V)$:
 \[ \cdots \longrightarrow (S_\q(V))^e \otimes \textstyle \bigwedge^2_{\q}(V) \overset{d_2}{\longrightarrow}
    (S_\q(V))^e \otimes \textstyle \bigwedge^1_{\q}(V) \overset{d_1}{\longrightarrow} (S_\q(V))^e
    \overset{m}{\longrightarrow} S_\q(V) \longrightarrow 0, \]
 where 
 \[ d_m(1^{\otimes 2} \otimes w_{j_1} \wedge \dots \wedge w_{j_m}) = \]
 \[ \sum_{i=1}^m (-1)^{i+1} \left[ \left( \prod_{s=1}^i q_{j_s,j_i} \right) w_{j_i} \otimes 1 -
    \left( \prod_{s=i}^m q_{j_i,j_s} \right) \otimes w_{j_i} \right] \otimes w_{j_1} \wedge \dots \wedge
    w_{j_{i-1}} \wedge w_{j_{i+1}} \wedge \dots \wedge w_{j_m} \]
 whenever $1 \leq j_1 < \dots < j_m \leq k$. We refer the reader to~\cite{WAMBST} for
 more details on this construction.
 
 In~\cite{NAIDUWITH}, chain maps $\Psi_i$ are introduced between the bar resolution and 
 the quantum Koszul resolution of $S_\q(V)$.
 In particular, we have the map
 \[ \Psi_2 : (S_\q(V))^{\otimes 4} \longrightarrow (S_\q(V))^e \otimes \textstyle \bigwedge^2_{\q}(V) \]
 such that
 \begin{equation} \label{psi2}
  \Psi_2 (1 \otimes w_i \otimes w_j \otimes 1) = 
  \begin{cases}
   1 \otimes 1 \otimes w_i \wedge w_j, & \text{ for } 1 \leq i < j \leq k, \\
   0, & \text{ for } i \geq j.
  \end{cases}
 \end{equation}
 We will also need the following two maps:
 \begin{align*} 
  R_2 : \mathrm{Hom}_\fie \left( (S_\q(V))^{\otimes 2}, S_\q(V)\#G \right) & \longrightarrow 
        \mathrm{Hom}_\fie \left( (S_\q(V))^{\otimes 2}, S_\q(V)\#G \right)^G \\
  R_2(\gamma) & = \frac{1}{|G|} \sum_{g \in G} g(\gamma)
 \end{align*} 
 and
 \begin{align*}
  \theta_2^* : \mathrm{Hom}_\fie \left( (S_\q(V))^{\otimes 2}, S_\q(V)\#G \right)^G & \longrightarrow 
               \mathrm{Hom}_\fie \left( (S_\q(V)\#G)^{\otimes 2}, S_\q(V)\#G \right) \\
  \theta_2^*(\gamma)(a g \otimes b h) & = \gamma (a \otimes g(b)) \, g h.
 \end{align*}
 As discussed in~\cite{NAIDUWITH}, since $\Psi_2$ may not preserve the action of $G$,
 the map $R_2$ ensures $G$-invariance of the image. The map $\theta_2^*$ extends a function 
 defined on $(S_\q(V))^{\otimes 2}$ to a function defined on $(S_\q(V)\#G)^{\otimes 2}$.

 By~\cite[Theorem~$3.5$]{NAIDUWITH}, the composition $\theta_2^* \; R_2 \; \Psi_2^*$ induces
 an isomorphism 
 \[ \left( \bigoplus_{g \in G} \mathrm{HH}^2 (S_\q(V), (S_\q(V))_g) \right)^G \longrightarrow \;
    \mathrm{HH}^2 (S_\q(V)\#G). \]
 As a consequence, we get that if $\kappa \in (S_\q(V)\#G) \otimes \textstyle \bigwedge^2_{\q^{-1}}(V^*)$, then
 \begin{equation} \label{naiduwith}
    [\theta_2^* R_2 \Psi_2^*(\kappa)](w_i \otimes w_j) = 
    \frac{1}{|G|} \sum_{g \in G} g(\kappa(\Psi_2(1 \otimes g^{-1}(w_i) \otimes g^{-1}(w_j) \otimes 1)))
    \; \text{ for } i < j.
 \end{equation}

The following proposition is an explicit description of $\mu_1$. The idea of the proof is that
we will apply $\mu_1$ and $P_1 \; g_1 \; P_2 \; g_2 \otimes w_1^* \wedge w_2^*$ to $w_i \otimes w_j$
and show that we obtain the same result.

\begin{prop} \label{prop5.6}
 The map $\mu_1$ can be identified with $P_1 \; g_1 \; P_2 \; g_2 \otimes w_1^* \wedge w_2^*$.
\end{prop}

\begin{proof}
 Set $\kappa = P_1 g_1 P_2 g_2 \otimes w_1^* \wedge w_2^*$ in~\eqref{naiduwith}. Then we have that
 \[ [\theta_2^* R_2 \Psi_2^*(P_1 g_1 P_2 g_2 \otimes w_1^* \wedge w_2^*)](w_i \otimes w_j) = \]
 \[ \frac{1}{|G|} \; \sum_{g \in G} g \left( P_1 g_1 P_2 g_2 (w_1^* \wedge w_2^*) \left( 
    \Psi_2 \left(1 \otimes g^{-1} \left(w_i \right) \otimes g^{-1} \left(w_j \right) \otimes 1 \right) \right) \right). \]
 By~\eqref{psi2}, we may simplify this expression as follows:
 \[ \frac{1}{|G|} \; \sum_{g \in G} g \left( P_1 g_1 P_2 g_2 (w_1^* \wedge w_2^*) \left(
    1 \otimes 1 \otimes g^{-1} \left(w_i \right) \wedge g^{-1} \left(w_j \right) \right) \right). \]
 Since $g^{-1}(w_i) = \chi_i^{-1}(g) \; w_i$ and applying~\eqref{EQU6}, this becomes
 \[ \frac{1}{|G|} \; \sum_{g \in G} \xi(g_1^{-1}) \; \chi_2(g_1) \; g \left( P_1 P_2 \# g_1 g_2 (w_1^* \wedge w_2^*)
    \left( 1 \otimes 1 \otimes \chi_i^{-1}(g) \; w_i \wedge \chi_j^{-1}(g) \; w_j \right) \right). \]
 By linearity, we get
 \[ \frac{1}{|G|} \; \xi(g_1^{-1}) \; \chi_2(g_1) \; \sum_{g \in G} \chi_i^{-1}(g) \; \chi_j^{-1}(g) \; g \left( P_1 P_2 \# g_1 g_2 (w_1^* \wedge w_2^*) 
    \left( 1 \otimes 1 \otimes w_i \wedge w_j \right) \right). \]
 If we assume that $i < j$, then
 \[ (w_1^* \wedge w_2^*)(w_i \wedge w_j) =
    \begin{cases}
     1, & \text{ if } i = 1 \text{ and } j = 2, \\
     0, & \text{ otherwise. }
    \end{cases}
 \]
 Thus, letting $i=1$ and $j=2$, the above expression becomes
 \[ \frac{1}{|G|} \; \xi(g_1^{-1}) \; \chi_2(g_1) \; \sum_{g \in G} \chi_1^{-1}(g) \; \chi_2^{-1}(g) \; g( P_1 P_2 \# g_1 g_2) = \]
 \[ \frac{1}{|G|} \; \xi(g_1^{-1}) \; \chi_2(g_1) \; \sum_{g \in G} \chi_1^{-1}(g) \; \chi_2^{-1}(g) \; g(P_1) \; g(P_2) \# g_1 g_2 = \]
 \[ \frac{1}{|G|} \; \xi(g_1^{-1}) \; \chi_2(g_1) \; \sum_{g \in G} \chi_1^{-1}(g) \; \chi_2^{-1}(g) \; \chi_1(g)\; \xi(g) \; P_1 \; \chi_2(g) \; \xi(g^{-1}) \; P_2 \# g_1 g_2 \]
 by~\eqref{EQU5} and~\eqref{EQU6}. Simplifying, we obtain
 \[ \xi(g_1^{-1}) \; \chi_2(g_1) \; P_1 \; P_2 \# g_1 g_2 = P_1 \; g_1 \; P_2 \; g_2. \]
Therefore, we have shown that
\[ [\theta_2^* R_2 \Psi_2^*(P_1 g_1 P_2 g_2 \otimes w_1^* \wedge w_2^*)](w_i \otimes w_j) = 
   \begin{cases}
    P_1 \; g_1 \; P_2 \; g_2, & \text{ if } i=1 \text{ and } j=2, \\
    0, & \text{ otherwise. }
   \end{cases} \]
On the other hand, since $\mu_1 = m \circ (D_1 \otimes D_2)$, we have that
\[ \mu_1(w_i \otimes w_j) = D_1(w_i) \; D_2(w_j) =
\begin{cases}
 P_1 \; g_1 \; P_2 \; g_2, & \text{ if } i = 1 \text{ and } j = 2, \\
 0, & \text{ otherwise. }
\end{cases}
\]
Therefore, $\mu_1$ can be identified with $P_1 \; g_1 \; P_2 \; g_2 \otimes w_1^* \wedge w_2^*$.
\end{proof}

Since $P_1$, $g_1$, $P_2$, $g_2$, $w_1$ and $w_2$ are nonzero, we know that $P_1 \; g_1 \; P_2 \; g_2 \otimes w_1^* \wedge w_2^*$
is a nonzero Hochschild $2$-cocycle. Therefore, we have obtained the following:
\begin{thm} \label{nontrivial}
 When $G$ is finite and $P_1$ and $P_2$ are elements of the subalgebra of $S_\q(V)$
 generated by $\{ w_3, \dots, w_k \}$, all the deformations that result from 
 Theorem~\ref{thm3} and~\eqref{formula} are nontrivial.
\end{thm}

\begin{acknow}
 Most of the results in this work are part of my doctoral dissertation.
 I would like to thank my advisor, Dr. Sarah Witherspoon, for all her guidance, instruction and support.
\end{acknow}


\bibliographystyle{plain}
\bibliography{biblio}

\end{document}